\documentclass[10pt]{article}
 \usepackage{amsmath,amsfonts,amssymb,amsthm}
 \usepackage{hyperref}
 \usepackage[affil-it]{authblk} 
\usepackage{etoolbox}
\usepackage{lmodern}

\makeatletter
\patchcmd{\@maketitle}{\LARGE \@title}{\fontsize{16}{19.2}\selectfont\@title}{}{}
\makeatother

 \theoremstyle{plain}
 \newtheorem{theorem}{Theorem}[section]
 \newtheorem{lemma}[theorem]{Lemma}
 \newtheorem{corollary}[theorem]{Corollary}
 \newtheorem{proposition}[theorem]{Proposition}
 
  \newtheorem{question}[theorem]{Question}

 \theoremstyle{definition}
 \newtheorem{definition}[theorem]{Definition}
 \newtheorem{definitionnotation}[theorem]{Notation}

 \theoremstyle{remark}
 \newtheorem{remark}[theorem]{Remark}

\newcommand{\chara}{\operatorname{char}}
\newcommand{\iso}{\operatorname{Iso}}
\newcommand{\sections}{\operatorname{Sections}}

\newcommand{\projection}{\operatorname{pr}}

\newcommand{\mni}{\medskip\noindent}

\usepackage{graphicx}
\usepackage{xcolor}
\usepackage{float}
\usepackage{amssymb}
\usepackage{amsmath}
\usepackage{mathrsfs}
\usepackage{bm}
\usepackage[all,cmtip]{xy}
\usepackage{enumitem}
\usepackage[utf8]{inputenc}
\usepackage[english]{babel}
\usepackage{multirow}
\usepackage{mathtools}
\usepackage{scalerel,stackengine}
\usepackage{stmaryrd}
\usepackage{datetime}
\usepackage[title]{appendix}
\usepackage{url}
\usepackage{tabulary}
\usepackage{booktabs}

\title{Pseudo-N\'eron Model and Restriction of Sections II}

\title{%
  Pseudo-N\'eron Model and Restriction of Sections II  \\         
  \large Generalization, Examples and Applications}

\author{Santai Qu}

\date{\today}

\begin{document}

\maketitle

\begin{abstract}
This is the continuation of the article by the author that proves a broader class of families
admitting the theorem of restriction of sections other than Abelian varieties and
gives new examples of pseudo-N\'eron models.  
In this work, we show that the techniques in the first paper give more general 
results and more examples such that the theorem of restriction of sections holds.
Also, we give counter examples to show that the non-existence of rational curves
is not a necessary condition for such theorems.  
As an application, we prove a result for Hodge classes of a certain weight which is similar
to the result about sections.
\end{abstract}

\tableofcontents

\section{Introduction}

This is the second article studying the theorem of restriction of sections.
The starting point of this work is the theorem proved in \cite{Santai}, Theorem 1.3,
which will be shown to fit into a general theorem in this article.
For the history and motivation of the problem of restriction of sections, 
the readers are refered to \cite{Santai}, Section 1.2.

\subsection{Main results about restriction of sections}
To state our main results, we first fix our conventions about curves.

\begin{definition}\label{linesdef}
Let $k$ be an algebraically closed field.  Fix a generically finite, generically unramified morphism $u_0: S\rightarrow \mathbb{P}^n_k$.  We define
\begin{itemize}
\item{an \emph{$u_0$-smooth-curve} of genus-$g$, degree-$d$, is an irreducible smooth curve in $S$ of the form $S\times_{\mathbb{P}^n_k} C_0$ for an irreducible, smooth, genus $g$ and degree $d$ curve $C_0\subset \mathbb{P}^n_k$,}
\item{an \emph{$u_0$-curve-pair} of genus-$(g_0+g_1)$, degree-$(d_0+d_1)$ is a connected curve in $S$ of the form $S\times_{\mathbb{P}^n_k}C$, where $C=C_0\cup C_1$ is a pair of curves in $\mathbb{P}^n_k$ intersecting transversally at a single closed point such that
$C_0$ (resp. $C_1$) is an irreducible, smooth curve of genus $g_0$ (resp. $g_1$) and degree $d_0$ (resp. $d_1$).}
\end{itemize}
\end{definition}

Let $k$ be an uncountable algebraically closed field.  We say a subset of a scheme is \emph{general}, resp. \emph{very general}, if the subset contains an open dense subset, resp. the intersection of a countable collection of open dense subsets.  We say that a property of points in a scheme holds \emph{at a general point}, resp. \emph{at a very general point}, if the set where the property holds is a general subset, resp. a very general subset.

In the rest of this paper, we assume that
$k$ be is an algebraically closed field of characteristic zero,  
and that $u_0: S\rightarrow \mathbb{P}^n_k$ is a finite and \'etale morphism onto an open dense 
subset of $\mathbb{P}^n_k$.

In the work \cite{GJ}, Tom Graber and Jason Starr proved that a section of an Abelian scheme over $S$
can be detected by its restriction on a very general conic curve or a line-pair, where a 
line-pair is just a curve-pair of genus-0, degree-2 as in Definition~\ref{linesdef}.
The notation $\sections(A/S)$ will represent the set of sections of $A$ over $S$.

\begin{theorem}(\cite{GJ}, Theorem 1.3, p.312)\label{maintheoremjason}
Let $k$ be an uncountable algebraically closed field.  Let $S$ be an integral, normal, quasi-projective $k$-scheme of dimension $b\ge 2$.  Let $A$ be an Abelian scheme over $S$.  For a very general line-pair $C$ in $S$, the restriction map of sections
\[\sections(A/S)\to\sections(A_C/C)\]
is a bijection.  The theorem also holds with $C$ a very general planar surface in $S$.  If $\chara k=0$, this also holds with $C$ a very general conic in $S$.
\end{theorem}

However, when we pass to a finite cover of the Abelian scheme, 
we can not hope that Theorem~\ref{maintheoremjason} still holds for conic curves.
Instead, after increasing the degree of curves, the article \cite{Santai} proved the following theorem.

\begin{theorem}(\cite{Santai}, Theorem 1.9)\label{generalizedjasonmaintheorem}
Let $k$ an uncountable algebraically closed field of characteristic zero.  Let $S$ be an integral, normal, quasi-projective $k$-scheme of dimension $b\ge 2$.  Let $X$ be a smooth $S$-scheme admitting a finite morphism $f: X\rightarrow A$ to an Abelian scheme $A$ over $S$.  Let $e$ be the fiber dimension of $\iso(A)$ where $\iso(A)$ is the isotrivial factor of $A$ (see \cite{Santai}, Definition-Lemma 3.21, for the definition of $\iso(A)$).  Let $d$ be a positive even integer.  

Then, for $d>2e$, and a very general genus-0 and degree-$(d+2)$ curve-pair or a smooth curve $C$, 
the restriction of sections
\[\sections(X/S)\to \sections(X_C/C)\]
is a bijection. 
\end{theorem}

In this article, we give a generalization of Theorem~\ref{generalizedjasonmaintheorem}.  
Our theorem provides more examples
(cf. Theorem~\ref{generalizedjasonmaintheorem2} and Corollar~\ref{hypersurfaces})
than Abelian schemes and schemes admitting finite morphisms to Abelian schemes for the theorem of restriction of sections.

\begin{theorem}\label{generalizedjasonmaintheorem3}
Let $k$ an uncountable algebraically closed field of characteristic zero.  
Let $S$ be an integral, normal, quasi-projective $k$-scheme of dimension $b\ge 2$.  
Let $X$ and $Y$ be smooth, projective $S$-schemes with a finite morphism $f: X\to Y$.
Let $e$ be the fiber dimension of $Y\to S$.

Assume that $Y\to S$ satisfies the following condition:

\begin{enumerate}[label=(\roman*)]
\item{$Y$ has a normal pseudo-N\'eron model $\widetilde{Y}$ over $W$ (cf. Section~\ref{basicfactonneronmodels}),}
\item{every geometric fiber $Y_{\overline{s}}$, $s\in S$, does not contain any rational curve, and}
\item{for a very general genus-$g_Y$, degree-$d_Y$, $u_0$-curve $C$ in $S$, the restriction map of sections
\[\sections(Y/S)\to \sections(Y_C/C)\] 
is a bijection.}
\end{enumerate} 

Then, for $g=rg_Y+g_Y$ and $d=rd_Y+d_Y$ with $r>e$, 
and a very general genus-$g$ and degree-$d$ curve-pair or a smooth curve $C$, 
the restriction of sections
\[\sections(X/S)\to \sections(X_C/C)\]
is also a bijection.
\end{theorem}

\subsection{Main results on examples}

Even though Theorem~\ref{generalizedjasonmaintheorem3} gives more examples
than Abelian schemes and schemes admitting finite morphisms to Abelian schemes, the non-existence of 
rational curves is actually not necessary for the theorem of restriction of sections.
Every geometric fiber of an Abelian scheme or its finite covers does not admit
any rational curve.  However, there are examples which also satisfies the theorem of 
restriction of sections but every geometric fiber has infinitely many rational curves.

\begin{theorem}\label{examplerc1}
Let $k$ be an uncountable algebraically closed field of characteristic zero.
Let $S$ be an open dense subset of $\mathbb{P}_k^n$, $n\ge 2$, of codimension at least two.
Then, there exists families of quasi-projective surfaces $X\to S$ such that 
for a very general line-pair, or a very general smooth conic curve $C$ in $S$, 
the restriction of sections
\[\sections(X/S)\to \sections(X_C/C)\]
is bijective.  Moreover, for every $k$-point $s\in S$, $X_s$ admits infinitely many 
non-proper rational curves.
\end{theorem}

A more careful construction gives examples with infinitely many complete rational curves on 
the geometric fibers.

\begin{theorem}\label{examplerc2}
Let $k$ be an uncountable algebraically closed field of characteristic zero.
Let $U$ and $B$ be smooth, irreducible and quasi-projective $k$-schemes of dimension $\ge 2$.
Assume that $u_0: U\to \mathbb{P}_k^n$ is a generically finite and generically \'etale
morphism onto an open dense subset of $\mathbb{P}_k^n$.
Suppose that $f_0: B\to U$ is a finite, surjective morphism of degree two.

Then, there exists families of singular, projective Kummer surfaces $X\to B$
such that for a very general line-pair, or a very general smooth conic curve $C$ in $B$, 
the restriction of sections
\[\sections(X/B)\to \sections(X_C/C)\]
is bijective.  Moreover, for every $k$-point $s\in B$, $X_s$ admits infinitely many 
complete rational curves.
\end{theorem}

\subsection{Application in Hodge theory}

Now, we take the ground field $k$ as the field of complex numbers $\mathbb{C}$.
In \cite{najm}, a special case of Theorem~\ref{maintheoremjason} is proved by
using Hodge theory.  The idea is that we can separate the problem into two cases:
isotrivial Abelian scheme; Abelian scheme that does not contain any isotrivial part.
Then, Hodge theory and the theory of Lefschetz pencils give the proof when $A\to S$
is isotrivial.  

Let $p: A\to S$ be an Abelian scheme of relative dimension $n$
and consider the local system $\mathbb{V}=\text{R}^{2n-1}p_*\mathbb{Q}$. 
Assume that $S$ has a normal projective compactification $\overline{S}$
such that $S$ in $\overline{S}$ is of codimension at least two.
Then, the key point for the proof of the second case is the following lemma.

\begin{lemma}(\cite{najm}, Lemma 2.6)
Let $p: A\to S$ be an Abelian scheme that does not contain any isotrivial 
Abelian subschemes.  Denote by $A(S)$ the group of sections of $A$ over $S$.
Then, via the cycle class map, we can identify $A(S)\otimes_{\mathbb{Z}}\mathbb{Q}$
with the group of Hodge classes of type $(n,n)$ of $\text{H}^1(S, \mathbb{V})$.
\end{lemma}
The main application of Hodge theory in this problem is that, for a very general 
curve $C$ in $S$ with sufficiently large degree, the restriction map
\[\text{H}^1(S, \mathbb{V})\to \text{H}^1(C, \mathbb{V}|_C)\]
is a bijection on Hodge classes of type $(n,n)$ (cf. \cite{najm}, Proposition 4.1).
We prove the converse of this result in Section~\ref{applicationinhodgetheory}, not only for families of Abelian varieties.

\begin{theorem}\label{h1}
Let $S$ be a smooth, quasi-projective $\mathbb{C}$-variety of dimension $b\ge 2$.
Let $p: X\to S$ be a smooth, surjective and projective morphism of relative dimension $n$.
Denote by $\mathbb{V}$ the local system $\text{R}^{2n-1}p_*\mathbb{Q}$ on $S$.
Let $C$ be a curve in $S$ such that the restriction map of sections
\[\sections(X/S)\to \sections(X_C/C)\]
is bijective.  Then, the restriction map
\[\text{H}^1(S, \mathbb{V})\to \text{H}^1(C, \mathbb{V}|_C)\]
is a bijection on Hodge classes of type $(n,n)$.
\end{theorem}

\subsection{Outline of the paper}

One of the key technical ingredients of Theorem~\ref{generalizedjasonmaintheorem3} is
the application of pseudo-N\'eron models.
We will give a brief review for the basic definitions and results we need about pseudo-N\'eron models 
in Section~\ref{basicfactonneronmodels}.  One of the technical ingredients in the proof of 
Theorem~\ref{generalizedjasonmaintheorem3} is that there are at most countably many 
section when we fix the image of a point.  These results are proved in Section~\ref{discretenessforthespaceofsections}.
In Section~\ref{generalsetupforrestrictionofsections}, we give a complete proof 
of Theorem~\ref{generalizedjasonmaintheorem3}.  Many parts of the proof are 
the same as the proof of Theorem~\ref{generalizedjasonmaintheorem}, so we 
only keep track of the main ideas in the proof of this article.  
For completeness and the convenience of readers, we include the proof of the main theorem
even though that is literally the same as in \cite{Santai}.
The most complicated part is to construct the various parameter spaces, 
which will be done in Section~\ref{variousparameterspaces}.
Next, Section~\ref{resultsbasedonconstantfamilies} aims to give an
application of Theorem~\ref{generalizedjasonmaintheorem3}.  Thus, we
get more examples for the theorem of restriction of sections than families
of Abelian varieties and varieties admitting finite morphisms to Abelian varieties.
In Section~\ref{examplewithrationalcurves12}, we give the proof of 
Theorem~\ref{examplerc1} and Theorem~\ref{examplerc2}.
And, finally, in Section~\ref{applicationinhodgetheory}, we give the proof 
of Theorem~\ref{h1}.

\vspace{1em}
\noindent{\bf Acknowledgement:}  The author is very grateful to his advisor Prof. Jason Michael Starr 
for his support during the proof.  The author also thanks Qianyu Chen for the disscusion about Hodge theory.

\section{Preliminaries for Pseudo-N\'eron Models}\label{basicfactonneronmodels}

In this section, we give a quick review of the results about pseudo-N\'eron models,
see \cite{Santai} for details and proof.

\begin{definition}\label{defweak}(\cite{GJ} Definition 4.10)
Let $S$ be an integral, regular, separated, Noetherian scheme of dimension $b\ge 1$.  A \emph{flat}, finite type, separated morphism $X\rightarrow S$ has the \emph{weak extension property} if for every triple ($Z\rightarrow S, U, s_U$) of
\begin{enumerate}[label=(\roman*)]
\item a smooth morphism $Z\rightarrow S$,
\item a dense, open subset $U\subset S$,
\item and an $S$-morphism $s_U:Z\times_SU\rightarrow X_U$,
\end{enumerate}
there exists a pair $(V,s_V)$ of 
\begin{enumerate}[label=(\roman*)]
\item an open subset $V\subset S$ containing U and all codimension 1 points of S,
\item and an S-morphism $s_V: Z\times_SV\rightarrow X$ whose restriction to $Z\times_SU$ is equal to $s_V$.
\end{enumerate}
\end{definition}

\begin{definition}(\cite{Santai}, Definition 3.2)\label{defmodel}
Let $S$ be an integral, regular, separated, Noetherian scheme of dimension $b\ge 1$.  Let $K$ be the fraction field of $S$, and $X_K$ be a smooth, separated $K$-scheme of finite type.  A flat, finite type, separated $S$-scheme $X$ is called a \emph{pseudo-N\'eron model} of $X_K$ if $X_K$ is isomorphic to its generic fiber and $X$ satisfies the weak extension property as in Definition~\ref{defweak}.
\end{definition}

Note that we do not require a pseudo-N\'eron model is smooth over $S$.  
It will be unique up to a unique isomorphism if it is smooth, and it is 
just the usual notion of N\'eron model.  It is well-known that every
Abelian variety has N\'eron model (cf. \cite{BLR}).  But if we pass to a finite cover 
of an Abelian variety, we can not hope to get its N\'eron model, only
a pseudo-N\'eron model.

\begin{theorem}(\cite{Santai}, Theorem 3.5)\label{neronhigh}
Suppose that $S$ is an integral, regular, separated, Noetherian Nagata scheme of dimension $b\ge 1$ with fraction field $K$. 
Let $X_K$ be a smooth scheme admitting finite $K$-morphism to a smooth, separated variety $Y_K$ of finite type which has a normal pseudo-N\'eron model $Y$ over $S$.  Then $X_K$ has a normal pseudo-N\'eron model $X$ over $S$.  

The theorem also holds if $X_K$ and $Y_K$ are replaced by some $X_U$ and $Y_U$ defined over a dense open $U$ of $S$.
\end{theorem}

\section{Discreteness for The Space of Sections}\label{discretenessforthespaceofsections}

\begin{lemma}\label{sectionsovercurve}
Let $k$ be an algebraically closed field.  Let $C$ be a smooth projective curve.  
Denote by $\pi: X\to C$ a smooth projective morphism such that for a closed point $b\in C$, 
the fiber $X_b$ does not contain any rational curve.  Then, for every closed point $p\in X_b$, 
there are at most countably many sections of $\pi$ mapping $b$ to $p$.
\end{lemma}

\begin{proof}
Let $C'$ be an irrational smooth projective curve with a finite, flat morphism $\tau: C'\to C$.  
Denote $X'$ the fiber product $X\times_C C'$.  Let $b'\in C'$ be a closed point whose image is $b\in C$, 
and $p'\in X'_{b'}$ a closed point whose image is $p\in X_b$.  
Then, for every section $\sigma$ of $\pi$ mapping $b$ to $p$, 
the base change $\sigma'$ of $\sigma$ via $\tau$ is a section of $X'\to C'$ mapping $b'$ to $p'$.  
Moreover, since $\tau$ is a fppf morphism, the map of sections
\[\sections_b^p(X/C)\to \sections_{b'}^{p'}(X'/C')\]
is injective.  Thus, it suffices to consider the case when $C$ is an irrational curve.

Let $\sigma$ be a section of $\pi$ mapping $b$ to $p$.  If $\dim_{[\sigma]}\text{Mor}(C,X;\sigma|_{\{b\}})$ 
has positive dimension, then there exists a rational curve on the fiber $X_b$ by the bend-and-break lemma 
(\cite{Debarre}, Proposition 3.11, p.70).  Therefore, 
the dimension $\dim_{[\sigma]}\text{Mor}(C,X;\sigma|_{\{b\}})$ must be zero, 
so there are at most countably many sections in $\sections_b^p(X/C)$.
\end{proof}

\begin{proposition}(\cite{gliu}, Proposition 6.2, p.1234)\label{extensionnorationalcurves}
Let $S$ be a Noetherian regular integral scheme, with function field $K$.
Let $X\to S$ be a proper morphism such that no geometric fiber $X_{\overline{s}}$
contains a rational curve.  Then any $K$-rational point of the generic fiber 
of $X\to S$ extends to a section over $S$.
\end{proposition}

Let $k$ be an uncountable algebraically closed field of characteristic zero.  
Let $S$ be a $k$-variety admitting a finite \'etale morphism $u_0$ onto an open dense 
subset $U$ of $\mathbb{P}_k^n$.  Let $u: W\to \mathbb{P}_k^n$ be a smooth 
projective compactification of $S$ (\cite{Santai}, section 3.3.1, p.20).  
Note that the morphism $u$ is surjective, and $S$ is an open dense subset of $W$.  

\begin{lemma}\label{shrink}
Let $\pi: X\to S$ be a smooth projective morphism such that every geometric fiber of $\pi$ does not contain any rational curves.  Let $C$ be a smooth curve in $S$.  
Then, a section $\gamma$ of $X$ over $C$ is contained in a unique section of $X$ over $S$ if and only if
there exists an open dense subset $U\subset S$ such that $\gamma|_{C\cap U}$ is contained in a unique 
section of $X$ over $U$.
\end{lemma}

\begin{proof}
Suppose that there exists an open dense subset $U$ of $S$ such that $\gamma|_{C\cap U}$
is contained in a unique section $\sigma$ of $X$ over $U$.  By Proposition~\ref{extensionnorationalcurves},
$\sigma$ extends to a section $\widetilde{\sigma}$ of $X$ over $S$.
Since the restrictions $\widetilde{\sigma}|_C$ and $\gamma$ agree 
on an open dense subset of $C$, by separatedness, $\gamma$
is contained in $\widetilde{\sigma}$.  Also, by separatedness, 
$\widetilde{\sigma}$ is the unique section that contains $\gamma$.
\end{proof}

\begin{lemma}\label{sectionsovergeneralcurve}
Let $\pi: X\to S$ be a smooth projective morphism admitting a pseudo-N\'eron model.  
Suppose that a very general geometric fiber of $\pi$ does not contain any rational curves.  
Then, for a very general smooth $u_0$-curve $C$, a very general closed point $b\in C$ 
and an arbitrary closed point $p\in X_b$, there are at most countably many sections of $X$ over $C$ that map $b$ to $p$.
\end{lemma}

\begin{proof}
Since $C$ is a $u_0$-curve, it comes from a smooth curve $C_0$ in $\mathbb{P}_k^n$.  
Let $C'$ be the base change $W\times_{\mathbb{P}_k^n}C_0$ which is a smooth projective curve, 
and contains $C$ as an open dense subset.  Let $\widetilde{X}$ be a pseudo-N\'eron model 
of $X$ over an open dense subset $V$ of $W$.  Denote by $\widetilde{\pi}: \widetilde{X}\to V$ the structural morphism 
of the pseudo-N\'eron model.  Note that $V$ contains all the codimension one points of $W$.  
Take a very general smooth $u_0$-curve $C$ and a very general closed point $b\in C$ such that 
the fiber $X_b$ does not contain any rational curve.  Since $V$ is at least codimension two in $W$, 
we can take that the curve $C'$ does not intersect the closed subset $W\setminus V$, i.e. 
$C'$ is contained in $V$.  Let $\sigma$ be a section of $X$ over $C$ that maps $b$ to $p$.  
Since $C'$ is normal, $\sigma$ extends to a morphism $\sigma': C'\to \widetilde{X}$.  
The restriction of $\widetilde{\pi}\circ \sigma'$ on $C$ is identity.  By separatedness, 
the composition $\widetilde{\pi}\circ \sigma'$ is identity.  So $\sigma'$ is a section of $\widetilde{\pi}$ 
that maps $b$ to $p$.  By Lemma~\ref{sectionsovercurve}, there are at most countably many 
such sections $\sigma'$ that map $b$ to $p$.  
Therefore, also there are countably many section of $X$ over $C$ mapping $b$ to $p$.
\end{proof}

\begin{lemma}\label{sectionsoverbase}
Let $\pi: X\to S$ be a smooth projective morphism admitting a pseudo-N\'eron model.  
Suppose that a very general geometric fiber of $\pi$ does not contain any rational curves.  
Then, for a very general closed point $b\in S$ and an arbitrary closed point $p\in X_b$, 
there are at most countably many sections of $\pi$ that map $b$ to $p$.
\end{lemma}

\begin{proof}
Let $\Delta$ be the scheme parameterizing sections of $\pi$ that map $b$ to $p$.  
Suppose that there exists an irreducible component of $\Delta$ that has positive dimension. 
Then there exists an integral curve $T$ in $\Delta$.  Since every irreducible component of $\Delta$ is quasi-projective, 
also the curve $T$ is quasi-projective.  By taking the normalization of $T$, we can assume further that $T$ is smooth.  
Now, take the evaluation map 
\[\rho_{ev}: T\times_k S\to X.\]
Let $t_0$ and $t_1$ be two closed points in $T$.  The images $\rho_{ev}(\{t_0\}\times S)$ and $\rho_{ev}(\{t_1\}\times S)$ 
are distincet closed subschemes in $X$, so their intersection $A$ maps, via $\pi$, to a proper closed subscheme $B$ of $S$.  
Note that $A$ contains $p$, and the point $b$ is in $B$.  

Take a very general, irreducible smooth $u_0$-curve $C$ in $S$ such that $C$ contains $b$ and $C$ is not contained in $B$.  
Denote by $C_0$ the complement $C\setminus (B\cap C)$.  
Let $c_0$ and $c_1$ be two closed points in $C_0$.  
Assume that $\rho_{ev}(\{t_0\}\times\{c_0\})$ and $\rho_{ev}(\{t_1\}\times\{c_1\})$ are same in $X$.  
Then the closed points $c_0$ and $c_1$ are in $B$, contradicting our choices of $c_0$ and $c_1$.  
So $\rho_{ev}(\{t_0\}\times C)$ and $\rho_{ev}(\{t_1\}\times C)$ are distinct curves in $X$.  
Thus, $\{\rho_{ev}(\{t\}\times C)\}_{t\in T}$ gives a positive dimensional family of curves in $X$ such that every curve in the family passes through $p$.  
By the same argument as in Lemma~\ref{sectionsovercurve} and Lemma~\ref{sectionsovergeneralcurve}, 
there is a rational curve on the fiber $X_b$, contradicting the choice of $b$.  
Therefore, $\Delta$ can not be positive dimensional, so it is a countable set of isolated points.
\end{proof}

\section{General Set Up for Restriction of Sections}\label{generalsetupforrestrictionofsections}

In this section, we give a proof of Theorem~\ref{generalizedjasonmaintheorem3},
which gives the machinary to form more examples of the theorem of restriction of sections.

Fix an uncountable algebraically closed field $k$ of characteristic zero.
Let $u_0: S\rightarrow \mathbb{P}^n_k$ be a finite and \'etale morphism onto an open dense 
subset of $\mathbb{P}^n_k$.  Denote by $W$ a smooth, projective compactification of $S$ (\cite{Santai}, Notation 3.18).

\subsection{Inputs of the theorem}\label{inputsofthetheorem}

The input is a finite morphism $f: X\to Y$, where $\rho_X: X\to S$ and $\rho_Y: Y\to S$ are
smooth, projective morphisms satisfying the following conditions:

\begin{enumerate}[label=(\roman*)]
\item{$Y$ has a normal pseudo-N\'eron model $\widetilde{Y}$ over $W$,}
\item{every geometric fiber $Y_{\overline{s}}$, $s\in S$, does not contain any rational curve, and}
\item{for a very general genus-$g_Y$, degree-$d_Y$, $u_0$-curve $C$ in $S$, the restriction map of sections
\[\sections(Y/S)\to \sections(Y_C/C)\] 
is a bijection.}
\end{enumerate}
As immediate consequences, we have that every geometric fiber $X_{\overline{s}}$, $s\in S$, does not
contain any rational curve, and $X$ has a normal pseudo-N\'eron model over $W$ (Theorem~\ref{neronhigh}).

\subsection{Parameter spaces}\label{variousparameterspaces}

\subsubsection{Motivation of the construction}\label{motivationoftheconstruction}

We first define the bad set for sections and curve-pairs.  Fix a point $b\in S$, 
and let $p\in Y$ and $\sigma$ be a section of $Y$ over $S$ mapping $b$ to $p$.  
Denote $\mathfrak{m}=(C_0, q, C_1)$ be a curve-pair in $S$ of genus-$(g_0+g_1)$, 
degree-$(d_0+d_1)$, such that they intersect at a closed point $q\in S$.  

\begin{definitionnotation}\label{sectionsdefinition}
Let $\sections_b^p(X/S)$ be the set of sections of $X$ over $S$ 
such that every section in the set maps $b\in S$ to $p\in Y$ via $f: X\to Y$. 
Similarly, for a curve $C$ in $S$ passing through $b$, we can define $\sections_b^p(X_C/C)$, 
also $\sections_b^p(Y/S)$ and $\sections_b^p(Y_C/C)$.
\end{definitionnotation}  

\mni
As in \cite{Santai}, Section 3.3.1, we consider the following two properties,
\begin{enumerate}[label=(\roman*)]
\item $\sections_b^p(Y/S)\rightarrow\sections_b^p((Y\times_S\mathfrak{m})/\mathfrak{m})$ is bijective;
\item $\sections_b^p((X\times_{f, Y,\sigma}S)/S)\rightarrow \sections_b^p((X\times_{f, Y, \sigma}S\times_S C_1)/C_1)$ is bijective,
\end{enumerate}
where the maps of the sets of sections are restrictions and the fiber product 
$X\times_{f, Y,\sigma}S$ comes from the section $\sigma$ from $S$ to $Y$ mapping $b$ to $p$.
Note that for $\sections_b^p((X\times_{f, Y,\sigma}S)/S)$ we are talking about 
the top row of the following Cartesian diagram.

\[\xymatrix{
X\times_{f, Y,\sigma}S \ar[r]\ar[d]  &  S\ar[d]^{\sigma}   \\
X\ar[r]_f  &  Y  
}\]
Intuitively, the bad set will be 
\[\{(p,\sigma),(C_0, q, C_1)|\text{either (i) is false or (ii) is false}\}.\]
Fix a closed point $b\in S$.  
The main reason we consider this parameter space is the matching condition lemma (\cite{Santai}, Lemma 3.51),
which says that away from the bad sets the theorem of restriction of sections
holds if we fix the points $b$ and $p$.

\begin{lemma}(\cite{Santai}, Lemma 3.51)\label{matching}
(1). Fix a point $b\in S$, a point $p$ in $Y_b$, and a point $p'\in f^{-1}(p)$.  
Let $\sigma$ be a section in $\sections_b^p(Y/S)$.  
Then, for a genus-$g_Y$, degree-$d_Y$ curve $C$
and a genus-$g$, degree-$d$ curve $m$ containing $b$ with $d\ge 2$ 
such that $C$ and $m$ intersect at a very general point, 
every section in $\sections_b^{p'}(X_{C\cup m}/C\cup m)$ that maps to $\sigma|_{C\cup m}$ is the restriction of a unique section in 
$\sections_b^{p'}(X/S)$ that maps to $\sigma$ if $C\cup m$ is good for $\sigma$ (see Definition~\ref{goodcurves}).

\mni
(2). Conversely, if $p$ is a bad point for $C\cup m$, then for $C$ and $m$ intersecting at a very general point, there exists a section in 
$\sections_b^{p'}(X_{C\cup m}/C\cup m)$ that cannot be extended uniquely.

\mni
(3). Let $C$ be a genus-$g$, degree-$d$, irreducible smooth curve marked by $b$ with $d\ge 2$.  
Then, every section in $\sections_b^{p'}(X_{C}/C)$ that maps to $\sigma|_{C}$ is the restriction of 
a unique section in $\sections_b^{p'}(X/S)$ that maps to $\sigma$ if $C$ is good for $\sigma$ (see Definition~\ref{goodirreduciblecurves}).
\end{lemma}

\begin{proof}
The proof is the same as \cite{Santai}, Lemma 3.51.
\end{proof}

We see that the parameter space of bad sets consists of two factors:
the space of pairs $(p, \sigma)$ for a point $p\in Y$ and a section of $Y$ over $S$
that maps $b$ to $p$; the space of marked curve-pairs $(C_0, q, C_1)$ with $b\in C_0$; 
similarly, the space of marked smooth irreducible curves for (3) in Lemma~\ref{matching}.

\subsubsection{The space of curves and curve-pairs}\label{thespaceofcurvepairsandcurves}

\begin{definitionnotation}\label{md2}
Let $\mathcal{M}_{d+d_Y}^{g+g_Y}(\mathbb{P}^n_k, \tau)$ be the stack parameterizing 
pointed curve-pairs of genus-$(g+g_Y)$, degree-$(d+d_Y)$ in $\mathbb{P}^n_k$.  
These are 4-tuples for such a pair $(s, [C_0], t, [C_1])$ consisting of a point $s$ of $\mathbb{P}^n_k$, 
a smooth curve $C_0$ of degree $d$, genus $g$ that contains $s$, a point $t$ on $C_0$, 
and a genus $g_Y$ and degree $d_Y$ that contains $t$.  
\end{definitionnotation}

The marked point $s$ defines an evaluation morphism 
\[\rho_{ev}: \mathcal{M}_{d+d_Y}^{g+g_Y}(\mathbb{P}^n_k, \tau) \rightarrow \mathbb{P}^n_k, (s, [C_0], t, [C_1])\mapsto s.\]

\begin{definitionnotation}\label{md0b}
Denote by $\mathcal{M}_{d+d_Y}^{g+g_Y}(\mathbb{P}^n_k, \tau, b)$ the fiber $\rho_{ev}^{-1}(b)$,
i.e., the stack of pointed curve-pairs in $\mathbb{P}^n_k$ such that the marked point on $C_0$ is $b$. 
\end{definitionnotation}

\begin{definitionnotation}\label{xxx}
Delete from $\mathcal{M}_{d+d_Y}^{g+g_Y}(\mathbb{P}^n_k, \tau, b)$  the closed subset 
parameterizing curve-pairs in which $C_0$ and $C_1$ are tangent to each other.
Then, let $\mathcal{X}$ be the subspace of $\mathcal{M}_{d+2}(\mathbb{P}_k^n, \tau, b)$ after this deletion.  
\end{definitionnotation}

\begin{definitionnotation}\label{md0bb}
Denote by $\mathcal{M}_d^g(\mathbb{P}^n_k, b)$ the stack of genus-$g$, 
degree-$d$ curves with a marked point $b$ in $\mathbb{P}^n_k$.
\end{definitionnotation}

\begin{definitionnotation}\label{mdvb}
Let $\mathcal{H}=\mathcal{M}_{d+d_Y}^{g+g_Y}(\mathbb{P}^n_k, \varepsilon, b)$ be the space parameterizing
genus $g+g_Y$, degree $d+d_Y$ curve with a marked point $b$, 
which can be reducible but at worst a curve-pair of genus-$(g+g_Y)$, degree-$(d+d_Y)$.  
And, if the curve is a curve-pair, the marked point $b$ is on the irreducible component 
that is not of genus-$(g+g_Y)$, degree-$(d+d_Y)$. 
\end{definitionnotation}

Recall that $X$ is a scheme over $S$ such that it admits a normal pseudo-N\'eron model $\widetilde{X}$ 
over an open dense $\widetilde{S}$, containing $S$, of codimension at least two in $W$.

\begin{definitionnotation}\label{xxx0}
As subspaces of $\mathcal{H}$ and $\mathcal{X}$,
denote the open locus of genus-$(g+g_Y)$, degree-$(d+d_Y)$ curves or curve-pairs 
contained in $\widetilde{S}$ by $\mathcal{H}_0$ and $\mathcal{X}_0$ respectively.  
Similarly, as subspaces of $\mathcal{H}$ and $\mathcal{X}$, let the space of genus-$(g+g_Y)$, degree-$(d+d_Y)$ 
curves or curve-pairs in $S$ be $\mathcal{H}_1$ and $\mathcal{X}_1$ respectively.  
\end{definitionnotation}

\begin{remark}
Taking the maximal open subschemes of the spaces $\mathcal{H}$, $\mathcal{X}$, 
$\mathcal{H}_0$, $\mathcal{X}_0$, $\mathcal{H}_1$ and $\mathcal{X}_1$, we can 
assume that they are all schemes over $k$ since we consider only very general curves and curve-pairs.  
Moreover, we may assume further that the scheme $\mathcal{H}$ is an integral
and smooth $k$-scheme and the locus $\mathcal{X}$ is an integral, smooth
Cartier divisor in $\mathcal{H}$.
\end{remark}

\subsubsection{The space of sections}\label{thespaceofsections}

Take the projective compactification $W$ of $S$ as in \cite{Santai}, Notation 3.18.
Let $\widetilde{X}\to \widetilde{S}$ be a pseudo-N\'eron model of $X\to S$ where $\widetilde{S}$ is an open dense subset 
of $W$ that contains all the codimension one points of $W$.
Let $\overline{X}\to W$ be a projective morphism whose restriction over $\widetilde{S}$ equals $\widetilde{X}$.

Because there is no rational curve on every geometric fiber of $X\to S$, 
by using Chow variety (see the construction in \cite{Santai}, Notation 3.25), we get a space $\mathfrak{H}$ parameterizing 
pairs $(p, \sigma)$ where $p\in Y_b$, and $\sigma$ is a section of $Y\to S$
that maps $b$ to $p$.  Note that every irreducible component of $\mathfrak{H}$ 
is a quasi-projective $k$-variety, but $\mathfrak{H}$ can have infinitely many irreducible components.
Also, there is a morphism $\mathfrak{H}\to Y$
mapping $(p, \sigma)$ to $p$.  

For a fixed closed point $p\in Y$, 
we denote the fiber of $\mathfrak{H}\to Y$ by $\mathfrak{H}_p$.  
Then, the space $\mathfrak{H}_p$ parameterizes sections of $Y$ over $S$
that maps $b$ to $p$.  By Lemma~\ref{sectionsoverbase},
$\mathfrak{H}_p$ is a countable union of isolated points.
However, it is possible that the local ring over an isolated point of $\mathfrak{H}_p$
is not the field $k$, but an Artinian local ring, so non-reduced.

\begin{definitionnotation}\label{remarkonmodulispaceofsections}
Let $\mathfrak{S}$ be an irreducible component of $\mathfrak{H}$.
Let $\mathbb{P}^n_{\mathfrak{S}}$, resp. $S_{\mathfrak{S}}$, resp. $Y_{\mathfrak{S}}$, resp. $X_{\mathfrak{S}}$, be the base change $\mathfrak{S}\times_k\mathbb{P}^n_k$, resp. $\mathfrak{S}\times_k S$, resp. $\mathfrak{S}\times_k Y$, resp. $\mathfrak{S}\times_k X$.  Denote by $\projection_1$ and $\projection_2$ the projections from $S_{\mathfrak{S}}$ to $\mathfrak{S}$ and $S$ respectively.    
\end{definitionnotation}

Similarly, we can define the space of sections over curves and curve-pairs.

\begin{definitionnotation}\label{universalsectionlifting}
Denote by $H$ the scheme which is universal for the problem of lifting curves, 
which are parameterized by $\mathcal{H}_1$ (see Notation~\ref{xxx0}), from $S$ to $X$, \emph{and mapping the point $b$ to $p$}.  
For the rigorous construction of $H$, see \cite{Santai}, Remark 3.45.
Then, $H$ parameterizes the pairs $([C], \gamma)$ for $[C]\in \mathcal{H}_1$ and $\gamma$
is a section of $X$ over $C$ that maps $b$ to $p$ via $f: X\to Y$.
Note that for every $[C]$, the fiber of $H\to \mathcal{H}_1$ over $[C]$ is a 
countable union of isolated points (Lemma~\ref{sectionsovergeneralcurve}).
But we do not exclude the possiblity that the fibers of $H\to \mathcal{H}_1$ are not reduced.
\end{definitionnotation}

\subsubsection{Specialization of sections over curves}

After defining the parameter spaces of curves and sections in Section~\ref{thespaceofcurvepairsandcurves} and Section~\ref{thespaceofsections}, 
we can state the main application of pseudo-N\'eron models in the proof of
theorem of restriction of sections.

\begin{definitionnotation}\label{ch01}
Keep the notations in Section~\ref{thespaceofcurvepairsandcurves}.
Let $\mathcal{C}_{\mathcal{H}} \subset\mathcal{H}\times_k W$ be the universal family 
of genus-$(g+g_Y)$, degree-$(d+d_Y)$ curves over $\mathcal{H}$.
We denote the open subset of universal family of genus-$(g+g_Y)$, degree-$(d+d_Y)$ 
curves in $\widetilde{S}$ (resp. $S$) by $\mathcal{C}_{\mathcal{H}_0}$ (resp. $\mathcal{C}_{\mathcal{H}_1}$).  
\end{definitionnotation}

\mni
Let $\Phi$ be the composition of the structure morphism of $H$ over $\mathcal{H}_1$ 
and the open immersion $\mathcal{H}_1\rightarrow \mathcal{H}_0$.
Let $\varrho: H\times_{\mathcal{H}_0}\mathcal{C}_{\mathcal{H}_0}\rightarrow \widetilde{X}$
be the universal section
which factors through the inclusion $X\to \widetilde{X}$.  
Then, we get the following diagram.   

\[\xymatrix{
  &   &  \widetilde{X}\ar[d] \\
H\times_{\mathcal{H}_0}\mathcal{C}_{\mathcal{H}_0}\ar[d]\ar[r]\ar[rru]^{\varrho} & \mathcal{C}_{\mathcal{H}_0}\ar[r]\ar[d] & \widetilde{S}   \\
H \ar[r]^{\Phi} & \mathcal{H}_0   &
}\]
Note that $H$ may have infinitely many irreducible components.  
And every irreducible component of $H$ is quasi-projective and quasi-finite over $\mathcal{H}_0$ 
(see Lemma~\ref{sectionsovergeneralcurve} for the quasi-finiteness).  

The key application of pseudo-N\'eron models in the problem of restriction of sections 
is that pseudo-N\'eron model of $X$ implies that $\varrho$ is actually 
defined over an open dense subset of $\mathcal{X}$ (\cite{Santai}, Theorem 3.47).
In other words, sections over irreducible, smooth, genus-$(g+g_Y)$, degree-$(d+d_Y)$
curves specialize to sections over curve-pairs of genus-$(g+g_Y)$ and degree-$(d+d_Y)$.
Thus, if the theorem of restriction of sections holds for curve-pairs, 
we can deform the curve-pairs to irreducible curves such that 
the theorem of restriction of sections also holds for irreducible smooth curves.
And we can prove that this is definitely the case.

\begin{proposition}(\cite{Santai}, Corollary 3.50)\label{inductionI}
Fix $b\in S$ and $p\in Y_b$ closed points.  Suppose that for a very general genus-$(g_0+g_1)$, degree-$(d_0+d_1)$ curve-pair $C_0\cup C_1$ every section in $\sections_b^p(X_{C_0\cup C_1}/C_0\cup C_1)$ is the restriction of a unique section in $\sections_b^p(X/S)$.  Then,  for a very general genus-$(g_0+g_1)$, degree-$(d_0+d_1)$ irreducible smooth curve containing $b$, every section over this curve mapping $b$ to $p$ is the restriction of a unique section of $X$ over $S$.
\end{proposition}

\begin{proof}
The main ingredients of the proof is Theorem 3.47 in  \cite{Santai} and the discreteness of the paremeter spaces of sections 
once we fix the points $b\in S$ and $p\in Y$ (Lemma~\ref{sectionsovergeneralcurve} and Lemma~\ref{sectionsoverbase}).
The the proof follows from the same way as \cite{Santai}, Corollary 3.50. 
\end{proof}

\subsubsection{The space of bad sets}

\begin{definitionnotation}\label{mds}
Denote by $\mathcal{M}_{d+d_Y}^{g+g_Y}({\mathbb{P}^n_{\mathfrak{S}}}{/\mathfrak{S}}, \tau)$ 
the space parameterizing schemes over $\mathfrak{S}$ such that every 
geometric fiber is a pointed curve-pair of genus-$(g+g_Y)$, degree-$(d+d_Y)$ as in Notation~\ref{md2}.
\end{definitionnotation}

There is also an evaluation map 
\[\rho_{ev}^{\mathfrak{S}}: \mathcal{M}_{d+d_Y}^{g+g_Y}({\mathbb{P}^n_{\mathfrak{S}}}{/\mathfrak{S}}, \tau)\rightarrow \mathbb{P}^n_{\mathfrak{S}}.\]
Denote by $\lambda$ the section of $S_{\mathfrak{S}}\rightarrow \mathfrak{S}$
marking every closed fiber by the point $b\in S$.
Let $u_0^{\mathfrak{S}}$ be the base change of the morphism $u_0: S\to\mathbb{P}_k^n$ by $\mathfrak{S}$.
Then, $\mu(b)=u_0^{\mathfrak{S}}\circ \lambda$ is a section of $\mathbb{P}^n_{\mathfrak{S}}$ over $\mathfrak{S}$. 

\begin{definitionnotation}\label{mdbb}
Denote by $\mathcal{M}_{d+d_Y}^{g+g_Y}({\mathbb{P}^n_{\mathfrak{S}}}{/\mathfrak{S}}, \tau)_{\mu(b)}$ the fiber product of $\rho_{ev}^{\mathfrak{S}}$ and the section $\mu(b): \mathfrak{S}\rightarrow \mathbb{P}^n_{\mathfrak{S}}$ over $\mathbb{P}^n_{\mathfrak{S}}$.  
\end{definitionnotation}
Then, $\mathcal{M}_{d+d_Y}^{g+g_Y}({\mathbb{P}^n_{\mathfrak{S}}}{/\mathfrak{S}}, \tau)_{\mu(b)}$ 
parameterizes the pairs of sections and curve-pairs $\{(p, \sigma), (b, [C_0], t, [C_1])\}$ 
where $\sigma$ is a section of $X$ over $S$ mapping $b$ to $p$ 
and the marked point on the curve $[C_0]$ is the fixed point $b$ (cf. Notation~\ref{md2}).

\begin{definitionnotation}\label{md01}
Denote by $\mathcal{M}_d^g({\mathbb{P}^n_{\mathfrak{S}}}{/\mathfrak{S}}, 1)$ the 
space parameterizing families over $\mathfrak{S}$
such that every geometric fiber is a genus-$g$, degree-$d$ curve with a marked closed point in $\mathbb{P}^n_k$, $(s, [C_0])$.
\end{definitionnotation}
Take the morphism forgetting the genus-$g_Y$, degree-$d_Y$ curve $[C_1]$
\[\mathcal{M}_{d+d_Y}^{g+g_Y}({\mathbb{P}^n_{\mathfrak{S}}}{/\mathfrak{S}}, \tau)\rightarrow \mathcal{M}_d^g({\mathbb{P}^n_{\mathfrak{S}}}{/\mathfrak{S}}, 1), (s, [C_0], t, [C_1])\mapsto (s, [C_0]).\]
Consider the diagram where the right square is Cartesian and $\varphi$ is just the composition
of the two vertical morphisms.

\[\xymatrix{
\mathcal{M}_{d+d_Y}^{g+g_Y}({\mathbb{P}^n_{\mathfrak{S}}}{/\mathfrak{S}}, \tau)_{\mu(b)}\ar[r]\ar[d]\ar@/_5pc/[dd]_{\varphi}   &  \mathfrak{S}\ar[d]^{\mu(b)}   \\
 \mathcal{M}_{d+d_Y}^{g+g_Y}({\mathbb{P}^n_{\mathfrak{S}}}{/\mathfrak{S}}, \tau)\ar[r] _{\,\,\,\,\,\,\,\,\,\,\,\,\,\,\,\,\,\,\,\,\rho_{ev}^{\mathfrak{S}}}\ar[d]  &  \mathbb{P}^n_{\mathfrak{S}}  \\
\mathcal{M}_d^g({\mathbb{P}^n_{\mathfrak{S}}}{/\mathfrak{S}}, 1)  &  
}\]
Take a $k$-point $\alpha$ of $\mathfrak{S}$ which is a pair $(p, \sigma)$.  
The fiber of the morphism $\varphi$ over $\alpha$ gives a $k$-morphism
\[\varphi_{\alpha}: \mathcal{M}_{d+d_Y}^{g+g_Y}(\mathbb{P}^n_k, \tau, b)\rightarrow \mathcal{M}_d^g(\mathbb{P}^n_k, b)\]
which is the forgetful morphism.  
For every $k$-point $\beta$ of $\mathcal{M}_d^g(\mathbb{P}^n_k, b)$ 
corresponding to a genus-$g$, degree-$d$ curve $C_0$,
the fiber of $\varphi_{\alpha}$ is a Zariski open dense subset $U_{\beta}$ of the variety parameterizing genus-$g_Y$, degree-$d_Y$ curves in $\mathbb{P}^n_k$ that intersect the curve $C_0$.
We claim that there is a maximal open dense subset $V_{\beta}$ of $U_{\beta}$ 
such that the restrictions of sections 
\[\sections((X\times_{f, Y,\sigma}S)/S)\rightarrow \sections((X\times_{f, Y, \sigma}S\times_S C_1)/C_1)\]
on the genus-$g_Y$, degree-$d_Y$ curves $C_1$ are bijective.
This follows from a variation of Bertini's theorem since $f: X\to Y$ is finite (\cite{Santai}, Theorem 3.17).
We summarize all the objects in the following diagram.
\[\xymatrix{
V_{\beta}\ar@{^(->}[r]^{\begin{subarray}{c} \text{open} \\ \text{dense} \end{subarray}}\ar[rd] & U_{\beta}\ar[r]\ar[d]& \mathcal{M}_{d+d_Y}^{g+g_Y}(\mathbb{P}^n_k, \tau, b)\ar[r] \ar[d]_{\varphi_{\alpha}} & \mathcal{M}_{d+d_Y}^{g+g_Y}({\mathbb{P}^n_{\mathfrak{S}}}{/\mathfrak{S}}, \tau)_{\mu(b)}\ar[d]^{\varphi}  \\
& \beta \ar[r]& \mathcal{M}_d^g(\mathbb{P}^n_k, b)\ar[r]\ar[d] & \mathcal{M}_d^g({\mathbb{P}^n_{\mathfrak{S}}}{/\mathfrak{S}}, 1)\ar[d]  \\
& & \alpha \ar[r]  &  \mathfrak{S}
}\]
Let $\mathcal{V}_{\alpha}$ be the union of $V_{\beta}$ in $\mathcal{M}_{d+d_Y}^{g+g_Y}(\mathbb{P}^n_k, \tau, b)$.  
This is an open dense subset of $\mathcal{M}_{d+d_Y}^{g+g_Y}(\mathbb{P}^n_k, \tau, b)$.  
Take the union $\mathcal{U}$ of $\mathcal{V}_{\alpha}$ in 
$\mathcal{M}_{d+d_Y}^{g+g_Y}({\mathbb{P}^n_{\mathfrak{S}}}{/\mathfrak{S}}, \tau)_{\mu(b)}$, 
which is also open dense in $\mathcal{M}_{d+d_Y}^{g+g_Y}({\mathbb{P}^n_{\mathfrak{S}}}{/\mathfrak{S}}, \tau)_{\mu(b)}$.   
\begin{definitionnotation}\label{www}
Denote by $\mathcal{W}$ the open dense subset of $\mathcal{M}_{d+d_Y}^{g+g_Y}({\mathbb{P}^n_{\mathfrak{S}}}{/\mathfrak{S}}, \tau)$ parameterizing the genus-$(g+g_Y)$, degree-$(d+d_Y)$ curve-pairs $\mathfrak{m}$ such that the restriction of sections
\[\sections(Y/S)\rightarrow\sections((Y\times_S\mathfrak{m})/\mathfrak{m})\]
is bijective.  
For $g$ a multiple of $g_Y$ and $d$ a multiple of $d_Y$, such $\mathcal{W}$ exists by
Proposition~\ref{inductionI} and the argument of \cite{Santai}, Lemma 3.52.
\end{definitionnotation}

We define the bad set $\mathcal{D}_b$ as the complement of  
$\mathcal{U}\cap (\mathcal{W}\times_{\mu(b)}\mathfrak{S})$ in 
$\mathcal{M}_{d+d_Y}^{g+g_Y}({\mathbb{P}^n_{\mathfrak{S}}}{/\mathfrak{S}}, \tau)_{\mu(b)}$.  
Moreover, by construction, $\mathcal{D}_b$ parameterizes the pointed pairs 
$\{(p, \sigma), (b, [C_0], q, [C_1])\}$ such that either (i) is false or (ii) is false (cf. Section~\ref{motivationoftheconstruction}).  
There are two natural projections from $\mathcal{D}_b$ to $\mathfrak{S}$ and $\mathcal{M}_{d+d_Y}^{g+g_Y}(\mathbb{P}^n_k, \tau, b)$, i.e.,
\[\phi_1: \mathcal{D}_b\to \mathfrak{S},\,\,\,\{(p, \sigma), (b, [C_0], q, [C_1])\}\mapsto (p, \sigma),\]
\[\phi_2: \mathcal{D}_b\to \mathcal{M}_{d+d_Y}^{g+g_Y}(\mathbb{P}^n_k, \tau, b), \,\,\,\{(p, \sigma), (b, [C_0], q, [C_1])\}\mapsto (b, [C_0], q, [C_1]).\]
By projecting once more from $\mathfrak{S}$ to $Y$, we get a morphism
\[\phi_3: \mathcal{D}_b\to Y,\,\,\,\{(p, \sigma), (b, [C_0], q, [C_1])\}\mapsto  p.\]
Denote the fiber of $\phi_3$ over $p$ by $\mathcal{D}_b^p$.

\begin{definition}\label{badsets}
The set $\mathcal{D}_b$ constructed above is called \emph{the bad set of sections and curve-pairs marked by $b$}.  
For $\{(p, \sigma), (b, [C_0], q, [C_1])\}$ in $\mathcal{D}_b^p$, $p$ is called a \emph{bad point} for the curve-pair $(b, [C_0], q, [C_1])$, 
and $\sigma$ is called a \emph{bad section} for $(b, [C_0], q, [C_1])$.
\end{definition}

\begin{definition}\label{goodcurves}
Fix $b\in S$ and $p\in Y_b$ closed points.  A curve-pair $\mathfrak{m}=(b, [C_0], q, [C_1])$ with $b$ marked on $[C_0]$ is called \emph{good for a section $\sigma$ in $\sections_b^p(Y/S)$} if the following two properties hold
\begin{enumerate}[label=(\roman*)]
\item $\sections_b^p(Y/S)\rightarrow\sections_b^p((Y\times_S \mathfrak{m})/\mathfrak{m})$ is bijective,
\item $\sections_b^p((X\times_{f, Y,\sigma}S)/S)\rightarrow \sections_b^p((X\times_{f, Y, \sigma}S\times_S C_1)/C_1)$ is bijective,
\end{enumerate}
\end{definition}

\begin{definition}\label{goodirreduciblecurves}
Fix $b\in S$ and $p\in Y_b$ closed points.  An irreducible smooth curve $C$ with a marked point $b\in S$ is called \emph{good for a section $\sigma$ in $\sections_b^p(Y/S)$} if the following two properties hold
\begin{enumerate}[label=(\roman*)]
\item $\sections_b^p(Y/S)\rightarrow\sections_b^p((Y\times_S C)/C)$ is bijective,
\item $\sections_b^p((X\times_{f, Y,\sigma}S)/S)\rightarrow \sections_b^p((X\times_{f, Y, \sigma}S\times_S C)/C)$ is bijective,
\end{enumerate}
\end{definition}

\subsection{Proof of the main theorem}

Now, having set all the parameter spaces, we can give the proof of Theorem~\ref{generalizedjasonmaintheorem3}.
The proof is literally the same as the proof \cite{Santai}, Theorem 1.3.  We include the proof
here for completeness of the whole strategy.

\begin{proof}
For convienience, we call an irreducible, smooth, genus-$g_Y$, degree-$d_Y$ curve in $S$ a \emph{Y-curve}
since the data $(g_Y, d_Y)$ depends on the family $Y\to S$.
Denote by $\mathcal{M}_{r, b}$ (resp. $\mathcal{M}'_{r, b}$) the bad set for 
a genus-$rg_Y$, degree-$rd_Y$ curve (resp. curve-pair) containing $b\in S$.

Let $C_1$ be a very general $Y$-curve containing a very general point $b_1\in S$.  
Let $\mathcal{M}_{1,b_1}$ be the image of $\phi_3(\mathcal{D}_{b_1})$, i.e., the set of bad points $p_1\in Y_{b_1}$ for $C_1$.  
Then, $\mathcal{M}_{1,b_1}$ is a proper subset of $Y_{b_1}$.  
Take $C_2$ a very general $Y$-curve intersecting with $C_1$ at a very general point $c_2$, 
and a very general point $b_2$ on $C_2$. 
Denote by $\Delta_1(C_1\cup C_2,b_1)$ the union of the images of $C_1\cup C_2$ 
under bad sections $\sigma$ of $Y$ over $S$ mapping $b_1$ to some $p\in\mathcal{M}_{1,b_1}$.  
Then, $\Delta_1(C_1\cup C_2,b_1)$ is contained in $\rho_Y^{-1}(C_1\cup C_2)$ and $\Delta_1(C_1\cup C_2,b_1)\cap \mathcal{M}_{1,b_1}$ equals $\mathcal{M}_{1,b_1}$.    Define $\Delta_1(C_1\cup C_2,b_2)$ in the same way for points in $\mathcal{M}_{1,b_2}$.

By choosing $C_2$, $c_2$ and $b_2$ very generally, $\Delta_1(C_1\cup C_2,b_2)\cap Y_{b_1}$ 
will intersect $\mathcal{M}_{1,b_1}$ transversally.  Moreover, since $C_1\cup_{c_2}C_2$ is very general, we may assume that 
\[\sections(Y/S)\to\sections(Y_{C_1\cup_{c_2}C_2}/C_1\cup_{c_2}C_2)\]
are bijective by the same argument of \cite{Santai}, Corollary 3.53.  

Let $p$ be a point in $\mathcal{M}_{1,b_1}$, but not in $\Delta_1(C_1\cup C_2,b_2)$.  
Let $\sigma$ be a section in $\sections_{b_1}^p(Y/S)$.  
If $\sigma(b_2)$ does not belong to $\mathcal{M}_{1,b_2}$, $C_1\cup C_2$ is good for $(\sigma, b_1,p)$.  
If $\sigma(b_2)$ is in $\mathcal{M}_{1,b_2}$, then $\sigma(b_1)$ is in $\Delta_1(C_1\cup C_2,b_2)$, 
which contradicts the choice of $p$.  
Thus, $C_1\cup C_2$ is good for every section in $\sections_{b_1}^p(Y/S)$, and $p$ is a good point for this marked curve-pair.  
Now, take a point $p_1'$ in $f^{-1}(p)$ where $p$ is in the set $\Delta_1(C_1\cup C_2,b_2)\cap Y_{b_1}$, but not in $\mathcal{M}_{1,b_1}$.  Denote by $\gamma$ a section of $X$ over $C_1\cup C_2$ mapping $b_1$ to $p_1'$.  Let $p_2'=\gamma(b_2)$.  Denote by $p_1$, resp. $p_2$, the image of $p_1'$, resp. $p_2'$ in $Y$.  Let $\sigma$ be a section of $Y$ over $S$ extending $f\circ\gamma$.  Since $p_1$ is not in $\mathcal{M}_{1,b_1}$, $C_1\cup C_2$ is good for $(\sigma, b_2,p_2)$.  By Lemma~\ref{matching}, $\gamma$ extends to a unique section of $X$ over $S$ mapping $b_2$ to $p_2'$ and $b_1$ to $p_1'$.  Thus, by the second part of Lemma~\ref{matching}, $p$ is a good point for $(b_1,C_1,c_2,C_2)$.

Denote by $\mathcal{M}'_{2,b_1}$ the set of bad points of $(b_1,C_1,c_2,C_2)$.  Then, by the above argument, $\mathcal{M}'_{2,b_1}$ is contained in the intersection of $\Delta_1(C_1\cup C_2,b_2)$ and $\mathcal{M}_{1,b_1}$.  Therefore, $\dim\mathcal{M}'_{2,b_1}$ is strictly less than $\dim\mathcal{M}_{1,b_1}$.  Let $C_{1,2}$ be a very general, genus-$2g_Y$, degree-$2d_Y$, irreducible, smooth curve containing $b_1$.  By Corollary~\ref{inductionI} and Lemma~\ref{matching}, the bad set of points is contained in $\mathcal{M}'_{2,b_1}$.  Denote the bad points for $(C_{1,2},b_1)$ by $\mathcal{M}_{2,b_1}$.  Attach a very general $Y$-curve $C_3$ to $C_{1,2}$ at a very general point $c_3$.  Then inductively, we get a decreasing sequence of dimensions
\[\dim\mathcal{M}_{1,b_1}>\dim\mathcal{M}'_{2,b_1}\ge\dim\mathcal{M}_{2,b_1}>\dim\mathcal{M}'_{3,b_1}\ge\dim\mathcal{M}_{3,b_1}>\cdots.\]
Then, for $g=(r+1)g_Y$ and $d=(r+1)d_Y$ with $r>e$, the bad set for $(b_1,m,c,C)$ of genus-$g$, degree-$d$ is empty, hence every section of $X$ over $C\cup m$ is the restriction of a unique section.  And, by Corollary~\ref{inductionI}, also this is true for very general irreducible smooth curves of genus-$g$, degree-$d$.
\end{proof}

\begin{remark}
If there is a reasonable quotient of the scheme $Y\to S$, 
then we can hope that the theorem also holds for curves of lower degrees, 
see \cite{Santai}, Remark 3.25, for details.
\end{remark}

\section{Results Based on Constant Families}\label{resultsbasedonconstantfamilies}

Take $k$-varieties $A$, $B$ and their fiber product $A\times_k B$.
By a \emph{cross section} of the second projection $A\times_k B\to B$, 
we mean a section $\gamma$ of $A\times_k B\to B$ such that the image of $\gamma$
in $A\times_k B$ projects to a single point in $A$.

\begin{lemma}\label{crosssectionforcurves}
Let $k$ be an algebraically closed field of characteristic zero.  Let $X_0$ be a smooth, projective $k$-variety that does not contain any rational curve.  Let $C$ be a smooth, projective $k$-curve of genus zero.  Then, every section of $\projection_2: X_0\times_k C\to C$ is a cross section.
\end{lemma}

\begin{proof}
Let $\sigma$ be a section of $X_0\times_k C$ over $C$.  Then, $\sigma$ induces a morphism $\sigma_0: C\to X_0$. 
If $\sigma_0$ is constant, then $\sigma$ is a cross section.  
So we assume that $\sigma_0$ is not constant.  
Denote by $C_0$ the image of $\sigma_0$, an integral projective curve on $X_0$.  Let $\nu: C_0^{nor}\to C_0$ be the normalization morphism.  Then, $\sigma_0: C\to C_0$ factors through a unique finite morphism $\pi: C\to C_0^{nor}$ such that $\nu\circ\pi$ equals $\sigma_0$.  Since $C$ is of genus zero, also the genus of $C_0^{nor}$ is zero (\cite{Liuqing}, Cor. 7.4.19, p.291).  And since $k$ is algebraically closed, $C_0^{nor}$ is isomorphic to $\mathbb{P}_k^1$ (\cite{Liuqing}, Prop. 7.4.1, p.285).  As normalization is birational, the composition
\[\mathbb{P}_k^1\simeq C_0^{nor}\overset{\nu}{\xrightarrow{\hspace*{0.5cm}}} X_0\]
gives a rational curve on $X_0$ (\cite{Kollar}, Definition 2.6, p.105), contradicting that $X_0$ does not contain any rational curve.  Therefore, $\sigma_0$ is constant, and $\sigma$ is a cross section. 
\end{proof}

\begin{lemma}\label{crosssectionforbase}
Let $k$ be an algebraically closed field of characteristic zero.  Let $S$ be an open dense subset of $\mathbb{P}_k^n$.  Let $X_0$ be a smooth, projective $k$-variety that does not contain any rational curve.  Then, every section of $\projection_2: X_0\times_k S\to S$ is a cross section.
\end{lemma}

\begin{proof}
Let $b_1$ and $b_2$ be two arbitrary closed points in $S$.  There exists a unique line $L_{1,2}$ containing $b_1$ and $b_2$.  Denote by $\sigma_{1,2}$ the restriction of $\sigma$ on $L_{1,2}$.  Since $L_{1,2}$ is isomorphic to an open dense subset of $\mathbb{P}_k^1$, $\sigma_{1,2}$ extends uniquely to the whole $\mathbb{P}_k^1$.  Therefore, by the same argument as Lemma~\ref{crosssectionforcurves}, $\sigma_{1,2}$ is a cross section.  And since $b_1$ and $b_2$ are two arbitrary closed points in $S$, also $\sigma$ is a cross section.
\end{proof}

Then, Lemma~\ref{crosssectionforcurves} and Lemma~\ref{crosssectionforbase}
give the result of restriction of sections for constant families immediately
since the only sections we have are cross sections.

\begin{proposition}\label{forconstantfamilies}
Suppose that $k$ is an algebraically closed field of characteristic zero, 
$S$ is an open dense subset of $\mathbb{P}_k^n$, 
and $X_0$ is a smooth and projective $k$-variety that does not contain any rational curve.
Then, for every irreducible, smooth and genus-0 curve $C$ in $\mathbb{P}_k^n$, 
the restriction map of sections
\[\sections(X_0\times_k S/S)\to \sections(X_0\times_k C/C)\]
is a bijection.
\end{proposition}

Keep the hypothesis in Proposition~\ref{forconstantfamilies}, let $X$ be a smooth, projective $S$-scheme
admitting a finite morphism $f: X\to X_0\times_k S$.
Denote by $\pi$ the projection from $X_0\times_k S$ to $X_0$.  Let $\rho$ be the structrual morphism of $X$.
\[\xymatrix{
X\ar[r]^{f\,\,\,\,\,\,\,\,\,\,}\ar[rd]_{\rho}  &  X_0\times_k S\ar[r]^{\,\,\,\,\,\,\,\,\,\,\pi}\ar[d]  & X_0  \\
  &  S  & 
}\]

\begin{remark}\label{notverygeneral}
Let $b$ be a closed point of $S$ and $p$ be a closed point of $X_0$.  Since $f$ is finite, there are only finitely many sections of $X$ over $S$ mapping $b$ to $p$.  Similarly, for a smooth, genus-0, degree-$d$ curve $C$ in $S$ containing $b$, there are only finitely many sections of $X$ over $C$ mapping $b$ to $p$.
\end{remark}

Because Proposition~\ref{forconstantfamilies} does not require the very generalness 
of curves and Remark~\ref{notverygeneral} guarantees that we do not have to 
remove countably many closed subsets in $S$ when applying Bertini's theorem to any base change of $f$, 
we do not have to assume that the field $k$ is uncountable, and we can improve the very generalness
to generalness in the following theorem.

\begin{theorem}\label{generalizedjasonmaintheorem2}
Let $k$ an algebraically closed field of characteristic zero.  Let $S$ be an open dense subset of $\mathbb{P}_k^n$.  Let $X_0$ be a smooth, projective $k$-variety that does not contain any rational curve.  Let $X$ be a smooth $S$-scheme admitting a finite $S$-morphism $f: X\rightarrow X_0\times_k S$.  Let $e$ be the dimension of $X_0$.  

Then, for $d>e$, every section of $X_C$ over a general genus-0 and degree-$(d+1)$ curve-pair or a general genus-0, degree-$(d+1)$ smooth curve $C$ is the restriction of a unique global section of $X$ over $S$. 
\end{theorem}

\begin{proof}
This follows directly from the general result Theorem~\ref{generalizedjasonmaintheorem3}
and \cite{Santai}, Corollary 2.8.
\end{proof}

\begin{corollary}\label{hypersurfaces}
Let $k$ be an uncountable algebraically closed field.  Let $S$ be an open dense subset of $\mathbb{P}_k^n$.  Let $X_0$ be a very general hypersurface of degree $\ge 2N-1$ in $\mathbb{P}^N_k$.  Let $X$ be a smooth $S$-scheme admitting a finite $S$-morphism $f: X\rightarrow X_0\times_k S$.   

Then, for $d\ge N$, every section of $X_C$ over a general genus-0 and degree-$(d+1)$ curve-pair or a general genus-0, degree-$(d+1)$ smooth curve $C$ is the restriction of a unique global section of $X$ over $S$.
\end{corollary}

\begin{proof}
This follows from Theorem~\ref{generalizedjasonmaintheorem2} and \cite{nocurves}, Theorem 1.2.
Note that we have to assume that the field $k$ is uncountable to apply \cite{nocurves}, Theorem 1.2.
\end{proof}

\begin{remark}
This is the new example of pseudo-N\'eron models provided in \cite{Santai}, Theorem 1.9.
We see that the existence of pseudo-N\'eron model also gives the theorem of 
restriction of sections for this example.
\end{remark}

\section{Examples with Rational Curves}\label{examplewithrationalcurves12}

\subsection{Examples with non-proper rational curves}

In this section, we give the proof of Theorem~\ref{examplerc1}.
Suppose that $k$ is an uncountable algebraically closed field of characteristic zero.
Let $S$ be an open dense subset of $\mathbb{P}_k^r$, $r\ge 2$, of codimension at least two.
Let $A$ be an Abelian scheme over $S$ of relative dimension two, that is,
$A$ is a family of Abelian surfaces over $S$.  Denote by $\iota$ the involution of $A\to S$.

Denote by $Y$ the quotient $A/\iota$.  Then, every closed fiber of $Y\to S$ has only 16 rational 
double point singularities, which gives $Y\to S$ 16 sections as singular locus of $Y$.
Denote by $X$ the complement of these 16 sections in $Y$.
Let $A'$ be the complement in $A$ of the inverse image of these 16 sections via $A\to Y$.
Then, $A'\to X$ is an \'etale double cover.
Let $Y'$ be the minimal resolution of $Y$.  The fibers of the projective 
family $Y'$ are projective smooth Kummer surfaces.
We include the following diagram to clarify the situation.

\[\xymatrix{
  &  Y'\ar[d]^{\text{blowing\,\,\,up}}  \\
A\ar[r]^{\text{quotient}}   &  Y  \\  
A'\ar[r]_{\text{\'etale}}\ar@{^{(}->}[u]^{\text{open}}  &  X\ar@{^{(}->}[u]_{\text{open}}
}\]

\begin{proposition}\label{exampleswithraitonalcurves}
For a very general line-pair, or a very general smooth conic curve $C$ in $S$, 
the restriction map of sections
\[\sections(X/S)\to \sections(X_C/C)\]
is bijective.  

Moreover, for every $k$-point $s\in S$, $X_s$ admits infinitely many 
non-proper rational curves.
\end{proposition}

\begin{proof}
For every $k$-point $s\in S$, $Y'_s\to Y_s$ is the blowing up at 16 singularities, which gives 16 disjoint
rational curves on $Y'_s$.  By \cite{bogomolov}, Example 5, the Kummer surface $Y'_s$ admits an infinite number of rational curves.  
So there must be infinitely many rational curves on $Y'_s$ different from the 16 distinguished rational curves.  
Since blowing up is proper and birational, there are also infinitely many 
rational curves on $Y_s$.  However, every rational curve different from the 16 distinguished rational curves
has to intersect at least one of the 16 distinguished rationa curves.  Thus, there are
infinitely many non-proper rational curves on $X_s$.

Let $C$ be a very general smooth conic curve in $S$.  Since $S$ is an open dense subset of $\mathbb{P}_k^r$
with codimension at least two, $C$ is a complete rational curve in $S$.  Suppose that $\gamma$ is a section 
of $X_C/C$.  Consider the following Cartesian diagram.

\[\xymatrix{
A'_C\ar[r]\ar[d]  &  C\ar[d]^{\gamma}  \\
A'\ar[r]  &  X  
}\]
Since $A'_C\to C$ is an \'etale double cover and $C$ is isomorphic to $\mathbb{P}_k^1$,
$A'_C$ consists of two copies of $\mathbb{P}_k^1$, 
each of which is mapped identically onto $C$.
Thus, $\gamma$ lifts to two sections, $\gamma_1$ and $\gamma_2$, of $A'$ over $C$, hence two sections of $A_C/C$.
Every member of these two sections over $C$ is contained in a unique section of $A/S$, say, $\sigma_1$ and $\sigma_2$,
because $A/S$ is an Abelian scheme and $C$ is very general (\cite{GJ}, Theorem 1.3, p.312).  
The restrictions $\sigma_1|_{A'}$ and $\sigma_2|_{A'}$
are mapped to sections $\tau_1$ and $\tau_2$ of $X/S$, both of which contain $\gamma$.
The involution $\iota$ mappes $\gamma_1$ to $\gamma_2$.  By the uniqueness of the extended
global section, we have that $\iota(\sigma_1)$ equals $\sigma_2$.  As a result, the sections $\tau_1$
and $\tau_2$ are same.  Therefore, $\gamma$ is contained in the unique section $\tau_1=\tau_2$.

Now, suppose that $C$ is a line-pair.  Consider every component of $C$ and use the same argument as
above, the restriction map of sections is also bijective.
\end{proof}

\subsection{Examples with complete rational curves}

The example in this section proves Theorem~\ref{examplerc2}.
Let $k$ be an uncountable algebraically closed field of characteristic zero.
Let $U$ and $B$ be smooth, irreducible and quasi-projective $k$-schemes of dimension $\ge 2$.
Assume that $u_0: U\to \mathbb{P}_k^n$ is a generically finite and generically \'etale
morphism onto an open dense subset of $\mathbb{P}_k^n$.
Suppose that $f_0: B\to U$ is a finite morphism of degree two.

Denote by $A$ an Abelian scheme of relative dimension two over $B$.  
Let $A'$ be the Abelian scheme $A\times_B U$, which
has a finite morphism of degree two to $A$, i.e., $f: A'\to A$ is a two-to-one cover.
Let $X$ be the quotient of $A'$ by the involution action on $A'$.
Thus, by construction, every geometric fiber of $X\to B$ is a singular Kummer surface
with 16 rational double points.

Take a conic curve in $\mathbb{P}_k^n$ and denote its inverse image in $B$ (resp. $U$)
by $C$ (resp. $C_2$).  Let $\gamma$ be a section of $X$ over $C$.
Denote by the base change of $C$ (resp. $\gamma$) to $A'$ by $C'$ (resp. $\gamma'$).
So both $\gamma$ and $\gamma'$ are closed immersions and $C'\to C$ is
a two-to-one cover of curves.  
However, since $C'$ is not a curve in $B$, $\gamma'$ does not
give a section of $A'$ over a curve in $B'$.
We summarize the notations in the following diagram,
where the squares are Cartesian.

\[\xymatrix{
C\ar[d]_{\gamma}  &  C'\ar[l]\ar[d]_{\gamma'}  &  \\
X\ar[ddr]  &  A'\ar[l]\ar[r]^f\ar[dd]  &  A\ar[d]  \\
   &   &  B\ar[d]^{f_0}   \\
   &  B\ar[r]^{f_0}  &  U
}\]
Denote by $C_1$ the image curve $f\circ \gamma'(C')$,
which is mapped onto $C$ via $A\to B$.  
Moreover, since $C'\to C$ and $C'\to C_1$ are double covers, 
the curves $C_1$ and $C$ must be isomorphic.
Thus, we have a section $\gamma_1: C\to A$ such that the following diagram is Cartesian.

\[\xymatrix{
C'\ar[r]^{f\circ \gamma'}\ar[d]_{\gamma'}  &  C\ar[d]^{\gamma_1}  \\
A'\ar[r]^f  &  A
}\]
Now, since $C$ is a conic in $B$ (i.e., a line-pair or a smooth conic), we can take $C$ to 
be very general such that $\gamma_1$ is contained in a unique global section $\sigma_1$ of $A\to B$.
The base change $B'=A'\times_{f,A,\sigma_1} B$ is a closed subscheme of $A'$ that contains
the curve $C'$.  Note that $B'$ is mapped onto $B$ via $A'\to B$ and it is a double cover of $B$.

The involution of $A'$ restricts to an automorphism of $C'$.  So, also the involution of $A$ restricts
to an automorphism of $C$.  By the uniqueness of the extended global section, the involution on $A$
gives an automorphism of the image of $\sigma_1$ in $A$.  Therefore, the involution on $A'$ 
restricts to an automorphism of the closed immersion $B'\to A'$.  Then, the quotient $B'/\iota$
exists as a closed subscheme of $X$.  By construction, $B'/\iota$ is mapped isomorphically to $B$
via $A\to B$.  So, $B'$ gives a section of $X\to B$ that contains $\gamma$. 
Moreover, by chasing the diagram, such a global section must be unique.

\section{Application in Hodge Theory}\label{applicationinhodgetheory}

We work over the field of complex numbers in this section.
Let $p: X\to S$ be a smooth, projective morphism as in Theorem~\ref{h1}.
By the BBD's decomposition theorem, we have the uncanonical isomorphism

\[\text{R}p_*\mathbb{Q}\simeq \bigoplus_{i\in \mathbb{Z}} (\text{R}^ip_*\mathbb{Q})[-i].\]
Thus, there is an isomorphism in the category of mixed Hodge structures

\[\text{H}^{2n}(X,\mathbb{Q})=\text{H}^{2n}(S, \text{R}p_*\mathbb{Q})
\simeq \bigoplus_{i\in \mathbb{Z}} \text{H}^{2n-i}(S, \text{R}^ip_*\mathbb{Q}).\]
Note that $\text{Hdg}^n(X)$ is defined to be $\text{H}^{2n}(X,\mathbb{Q})\cap \text{H}^{n,n}(X)$.

Now, we give the proof of Theorem~\ref{h1}.

\begin{proof}
Let $\text{Z}(X)$ (resp. $\text{Z}(X_C)$) be the group of algebraic cycles on $X$ (resp. $X_C$) with coefficients in $\mathbb{Q}$.
Denote by $\text{Z}(X/S)$ (resp. $\text{Z}(X_C/C)$) the subgroup of $\text{Z}(X)$ (resp. $\text{Z}(X_C)$)
that is generated by the images of sections of $X\to S$ (resp. $X_C\to C$).
Denote by $\text{Hdg}^1(\mathbb{V})$ the Hodge classes of type $(n,n)$ in $\text{H}^1(S, \mathbb{V})$.
Similarly, we can define $\text{Hdg}^1(\mathbb{V}|_C)$.  Consider the following diagram,
\[\xymatrix{
\text{Z}(X/S)\ar[d]\ar[r]  &  \text{Z}(X_C/C)\ar[d]      \\
\text{Hdg}^1(\mathbb{V})\ar[r]  &  \text{Hdg}^1(\mathbb{V}|_C)
}\]
where the vertical maps are cycle class maps via Leray spectral sequence.
By Lefschetz's theorem on $(1,1)$-classes, the vertical maps are surjective.
Moreover, by hypothesis, the top horizontal row is a bijection, so the restriction map
$$\text{Hdg}^1(\mathbb{V})\to \text{Hdg}^1(\mathbb{V}|_C)$$
is surjective.  Note that this map is also injective (\cite{najm}, Remark 2.7), 
so we complete the proof.
\end{proof}

\section{Further Questions}

Here are several further questions.

\begin{question}
Keep the notations in Theorem~\ref{generalizedjasonmaintheorem}.  Assume that $S$ is an integral, normal, quasi-projective $k$-scheme of dimension $b\ge 2$ admitting a finite \'etale morphism onto an open dense subset of $\mathbb{P}_k^n$.  Does the result in Theorem~\ref{generalizedjasonmaintheorem} still hold?
\end{question}

The problem is that there might be sections of $X_0\times_k S$ that are not cross sections.  Let $u_0: S\to \mathbb{P}_k^n$ be a finite \'etale morphism.  Let $C$ be a smooth, genus zero curve.  However, $u_0^{-1}(C)$ is not necessarily of genus zero.  If $S$ can be covered by smooth, genus zero curves, then the result in Theorem~\ref{generalizedjasonmaintheorem} still hold since every section of $X_0\times_k S\to S$ is a cross section by the same argument of Lemma~\ref{crosssectionforbase}.  If $S$ can not be covered by smooth, genus zero curves, since there is no group structure on $X_0$, the operator of taking difference does not apply as in \cite{GJ}.  
In other words, the problem about restriction of sections depends on 
the morphism $X\to S$, existence of pseudo-N\'eron model and the base $S$.

\begin{question}\label{questionofexamples}
Theorem~\ref{examplerc1} gives an example that is quasi-projective and smooth.  
Theorem~\ref{examplerc2} gives the example that is projective but singular along
some sections.  Do we have examples $X\to S$ that is smooth, projective, 
there are rational curves on geometric fibers, and the theorem of restriction 
of sections is true?
\end{question}

\begin{question}\label{questionabouthodgetheory}
Theorem~\ref{h1} gives the result about Hodge classes of weight $(n,n)$.
Then, how about Hodge classes of other weights?  Does the theorem of 
restriction of sections also give the same results as in Theorem~\ref{h1}
for some Hodge classes of other weights?
\end{question}

\vspace{1em}

\noindent\small{\textsc{Mathematics Department, Stony Brook University, Stony Brook, NY, 11794} }

\noindent\small{Email: \texttt{santai.qu@stonybrook.edu}}

 \end{document}